\documentclass{amsart}

\usepackage[T1]{fontenc}
\usepackage{amsfonts,amsthm,mathtools,amsmath,amssymb,mathrsfs}
\usepackage{microtype}
\usepackage[hidelinks]{hyperref}
\hypersetup{
  pdftitle={Algebraic Cauchy transforms of real algebraic ovals: monodromy, topology, and degree},
  pdfauthor={Christian Hagg and Boris Shapiro},
  pdfkeywords={logarithmic potential, Cauchy transform, algebraic oval, monodromy, quadrature domain}
}
\usepackage{enumitem}
\allowdisplaybreaks
\newtheorem{theorem}{Theorem}[section]
\newtheorem{proposition}[theorem]{Proposition}
\newtheorem{lemma}[theorem]{Lemma}
\newtheorem{corollary}[theorem]{Corollary}
\theoremstyle{definition}

\theoremstyle{remark}
\newtheorem{remark}[theorem]{Remark}

\newcommand{\C}{\mathbb C}
\newcommand{\R}{\mathbb R}
\newcommand{\Z}{\mathbb Z}
\newcommand{\Q}{\mathbb Q}
\newcommand{\PP}{\mathbb P}
\newcommand{\bd}{\partial}
\newcommand{\ct}{\mathcal C}
\newcommand{\Res}{\operatorname{Res}}

\newcommand{\ev}{\operatorname{ev}}
\newcommand{\LogPot}{\mathcal U}

\title[Algebraic Cauchy transforms of algebraic ovals]
{Algebraic Cauchy transforms of real algebraic ovals}

\author{Christian H\"agg}
\address{Department of Mathematics, Stockholm University, SE-106 91 Stockholm, Sweden}
\email{hagg@math.su.se}
\author{Boris Shapiro}
\address{Department of Mathematics, Stockholm University, SE-106 91 Stockholm, Sweden}
\email{shapiro@math.su.se}

\subjclass[2020]{31A15, 30E20, 30F10, 14H50, 14P25}
\keywords{logarithmic potential, Cauchy transform, algebraic oval, Schwarz correspondence, monodromy, dividing real curve, quadrature domain}
\date{}
\dedicatory{To the memory of Vladimir Arnold}

\begin{document}

\begin{abstract}
Let $d\mu_\Omega=\frac1\pi\mathbf 1_\Omega\,dA$ be the area measure of a bounded planar domain normalized by the factor $1/\pi$, and let $\LogPot_\Omega$ be its logarithmic potential.  The exterior Cauchy transform $\ct_\Omega=2\partial_z\LogPot_\Omega$ records the complex gradient of the potential outside $\overline\Omega$.  Suppose that $\partial\Omega$ is a smooth oval of a real algebraic curve, and let $\gamma$ be its lift to the normalization $X$ of the relevant irreducible Schwarz correspondence.  Writing $\pi,\eta:X\to\PP^1$ for the two Schwarz coordinates, one has
\[
 \ct_\Omega(z)=\frac{1}{2\pi i}\int_\gamma\frac{\eta\,d\pi}{z-\pi}.
\]
Our main result gives a topological characterization of algebraicity.  If the augmentation representation of the geometric monodromy group of $\pi$ is irreducible---equivalently, if the monodromy is two-transitive---then $\ct_\Omega$ is algebraic over $\C(z)$ if and only if $\gamma$ separates $X$.  On the separating side, a residue formula expresses the transform as an incomplete trace of $\eta$, up to a rational fixed-pole term, and yields an exact orbit-degree formula; for monodromy $S_d$ and a side containing $k$ points of a generic fibre, the degree is $\binom dk$.  We also construct, for every $h\ge1$, a star-shaped real-algebraic Jordan domain whose Schwarz normalization has genus $2h$ and whose exterior Cauchy transform is algebraic but nonrational.  Thus algebraic exterior Cauchy transforms form a strictly larger class than rational quadrature-domain transforms.
\end{abstract}

\maketitle

\section{Introduction}

Let $\Omega\subset\C$ be a bounded domain and let
\[
 d\mu_\Omega(\zeta)=\frac1\pi\mathbf 1_\Omega(\zeta)\,dA(\zeta).
\]
Its logarithmic potential and Cauchy transform are
\begin{equation}\label{eq:potential-cauchy}
 \LogPot_\Omega(z)=\int_\C\log|z-\zeta|\,d\mu_\Omega(\zeta),
 \qquad
 \ct_\Omega(z)=\int_\C\frac{d\mu_\Omega(\zeta)}{z-\zeta}.
\end{equation}
Distributionally, $\Delta\LogPot_\Omega=2\mathbf 1_\Omega$ and $\bar\partial\ct_\Omega=\mathbf 1_\Omega$; on $\C\setminus\overline\Omega$ one has $\ct_\Omega=2\partial_z\LogPot_\Omega$.  Thus the exterior Cauchy transform is the complex gradient of the logarithmic potential generated by uniform area measure.

For a classical quadrature domain, the exterior potential agrees with that of a finite distribution supported inside the domain; equivalently, under the standard regularity assumptions, $\ct_\Omega$ is rational.  This potential-theoretic interpretation is one of the basic links between quadrature identities, inverse potential problems, Schwarz functions, and free-boundary methods; see, for example, \cite{AharonovShapiro,Davis,GustafssonSchottky,GustafssonInverse,Sakai,GustafssonShapiro}.  Recent work on Helmholtz quadrature domains and partial balayage illustrates the continuing role of these ideas in potential analysis \cite{KowLarsonSaloShahgholian,GardinerSjodin}.  In the present paper we remain in the classical logarithmic setting but replace rationality by the weaker requirement that $\ct_\Omega$ be algebraic.

A second motivation comes from the circle of problems on algebraic integrability of ovals originating in Newton's Lemma~XXVIII and revived by Vladimir Arnold and Victor Vassiliev.  Arnold and Vassiliev emphasized algebraicity and monodromy in Newton's problem \cite{ArnoldVassiliev}, while Vassiliev subsequently obtained higher-dimensional nonintegrability results by Picard--Lefschetz and reflection-group methods \cite{VassilievNewton}.  The functional studied here is different: instead of the volume cut off by a moving hyperplane, we consider the complex gradient of the exterior logarithmic potential of a fixed planar domain.  Nevertheless, the organizing question is the same: when does a geometrically defined integral have only finitely many analytic branches?

Assume now that $\partial\Omega$ is a smooth oval contained in the nonsingular locus of the real algebraic curve $P(z,\bar z)=0$.  We normalize the irreducible component of the complexification $P(z,w)=0$ that contains the analytic lift of the oval, and denote the resulting compact Riemann surface by $X$.  The coordinate functions are
\[
 \pi:X\longrightarrow\PP^1_z,
 \qquad
 \eta:X\longrightarrow\PP^1_w,
\]
and the lifted oval is denoted by $\gamma\subset X$.  Along $\gamma$ one has $\eta=\overline\pi$, and the Cauchy--Green formula gives
\begin{equation}\label{eq:intro-lift}
 \ct_\Omega(z)=\frac1{2\pi i}\int_\gamma\frac{\eta\,d\pi}{z-\pi}.
\end{equation}
We refer to the normalized component $(X,\pi,\eta)$ as the relevant Schwarz correspondence of the oval.

Cauchy-type integrals of algebraic functions were studied systematically by Pakovich, Roytvarf, and Yomdin \cite{PRY}; related moment problems on Riemann surfaces and for hyperelliptic Abelian integrals were investigated by Gavrilov and Pakovich \cite{GavrilovPakovich}.  The general criterion of \cite{PRY} characterizes algebraicity by finiteness of the combinatorial monodromy.  The new point here is that, for a closed algebraic oval and under an irreducibility hypothesis on the projection monodromy, this analytic finiteness condition has a direct topological meaning: it is equivalent to the lifted oval being null-homologous, or equivalently separating, on $X$.  In addition, on the separating side we turn the transform into an explicit incomplete trace and compute its algebraic degree.  This supplies effective geometric and algebraic information not contained in the general finiteness criterion.

The following theorem is the effective algebraic statement.  For a meromorphic function $f\in\C(X)$ set
\[
 I_{\gamma,f}(z)=\frac1{2\pi i}\int_\gamma\frac{f\,d\pi}{z-\pi}.
\]
The exterior Cauchy transform corresponds to $f=\eta$.

\begin{theorem}[Residues, incomplete traces, and degree]\label{thm:A}
Let $\pi:X\to\PP^1$ have degree $d$, let $f\in\C(X)$, and suppose that $\gamma=\bd\Sigma$ for an integral two-chain $\Sigma$, with $\gamma$ disjoint from the poles of $f\,d\pi$.  Let $n_\Sigma$ denote the locally constant integer multiplicity of $\Sigma$ on $X\setminus|\gamma|$.  Fix a regular base point $z_0\notin\pi(\gamma)$, away from the images of the poles of $f$, and label the local fibre branches $q_1(z),\ldots,q_d(z)$ near $z_0$.  Put $m_i=n_\Sigma(q_i(z_0))$.  Then, as a germ at $z_0$,
\begin{equation}\label{eq:A-intro}
 I_{\gamma,f}(z)=R_\Sigma(z)-\sum_{i=1}^d m_i f(q_i(z)),
\end{equation}
where $R_\Sigma\in\C(z)$ is the contribution of the fixed poles of $f\,d\pi$.

Let $L$ denote the Galois closure of $\C(X)/\C(\pi)$, and identify the geometric monodromy group $G\le S_d$ with $\operatorname{Gal}(L/\C(z))$ acting on the $d$ sheets.  The degree of $I_{\gamma,f}$ over $\C(z)$ is at most $|G\cdot m|$.  If the evaluation map on the augmentation module
\[
 \ev_f:\Bigl\{(a_i)\in\C^d:\sum_i a_i=0\Bigr\}\longrightarrow L,
 \qquad (a_i)\longmapsto\sum_i a_i f(q_i),
\]
is injective, then the degree is exactly $|G\cdot m|$.

In particular, if $G=S_d$, $f\notin\C(\pi)$, and the distinct values among $m_1,\ldots,m_d$ occur with multiplicities $\nu_1,\ldots,\nu_s$, then
\[
 \deg_{\C(z)} I_{\gamma,f}=\frac{d!}{\nu_1!\cdots\nu_s!}.
\]
For a bordered side $\Sigma$, if precisely $k$ points of a generic fibre lie in $\Sigma$, then
\[
 \deg_{\C(z)} I_{\gamma,f}=\binom dk.
\]
\end{theorem}

The converse is the principal result of the paper.  Since $X$ is irreducible, the geometric monodromy group is transitive; for a transitive permutation group, irreducibility of the complex augmentation representation is equivalent to two-transitivity.

\begin{theorem}[Topological algebraicity criterion]\label{thm:B}
Let $\Gamma\subset\C$ be a smooth Jordan curve and let $\gamma\subset X$ be a closed lift on which $\pi:\gamma\to\Gamma$ is a diffeomorphism.  Let $f\in\C(X)\setminus\C(\pi)$ be regular along $\gamma$, and assume that $f\,d\pi$ has no pole on $\gamma$.  Assume that the augmentation representation of the geometric monodromy group of $\pi$ is irreducible, equivalently that this group is two-transitive.  Then
\[
 I_{\gamma,f}\text{ is algebraic over }\C(z)
 \quad\Longleftrightarrow\quad
 [\gamma]=0\text{ in }H_1(X,\Z).
\]
Equivalently, $I_{\gamma,f}$ is algebraic if and only if $\gamma$ separates $X$.
For the exterior Cauchy transform one takes $f=\eta$.  Since $\C(X)=\C(\pi,\eta)$ on the relevant Schwarz component, the nontriviality hypothesis is automatic when $\deg\pi>1$.
\end{theorem}

For a nonsingular irreducible plane curve, a generic affine projection has simple finite branch points with distinct branch values.  Its local monodromies are transpositions; transitivity then forces the full symmetric group.  We obtain the following geometric consequence.

\begin{corollary}[Generic affine dichotomy]\label{cor:generic}
Let the relevant real plane curve be nonsingular and irreducible.  For a generic real affine coordinate, the exterior Cauchy transform of an oval is algebraic if and only if the lifted oval separates the normalization.

In particular:
\begin{enumerate}[label=(\roman*)]
\item for every nonsingular real cubic with a compact oval, the exterior Cauchy transform is nonalgebraic for a generic real affine coordinate;
\item if the normalized real locus has at least two components, the transform of each individual oval is generically nonalgebraic;
\item the unique oval of a nonsingular real plane curve of even degree at least four has generically nonalgebraic exterior transform.
\end{enumerate}
\end{corollary}

The final result shows that the algebraic case is substantially larger than the rational, quadrature-domain case.

\begin{theorem}[Algebraic but nonrational exterior Cauchy transforms]\label{thm:C}
For every $h\ge1$ and every real $\varepsilon$ with $0<|\varepsilon|<2$, there is a bounded star-shaped real-algebraic Jordan domain $\Omega_{h,\varepsilon}$ such that:
\begin{enumerate}[label=(\roman*)]
\item the normalization of its irreducible Schwarz correspondence has genus $2h$;
\item the lifted boundary is the complete real locus and separates the normalization;
\item $\ct_{\Omega_{h,\varepsilon}}$ is algebraic but nonrational.
\end{enumerate}
\end{theorem}

The examples are obtained from the hyperelliptic curves $y^2=x^{4h+2}+1$ by an explicit birational real projection.  Their parity is optimal: a dividing real curve with one real component necessarily has even genus.

The paper is organized as follows.  Section~\ref{sec:boundary} records the potential-theoretic normalization, the boundary formula, and the residue expression.  Section~\ref{sec:cocycle} identifies the branch-sum cocycle with intersection numbers of lifted loops.  Section~\ref{sec:proofs} proves Theorems~\ref{thm:A} and~\ref{thm:B}.  Section~\ref{sec:real} gives consequences for real algebraic ovals, and Section~\ref{sec:examples} proves Theorem~\ref{thm:C}.

\section{Logarithmic potentials and boundary reduction}\label{sec:boundary}

We first record the elementary potential-theoretic identities behind the boundary integral.

\begin{proposition}[Potential and moment representation]\label{prop:potential}
For a bounded measurable set $\Omega\subset\C$, the distributions in \eqref{eq:potential-cauchy} satisfy
\[
 \Delta\LogPot_\Omega=2\mathbf 1_\Omega,
 \qquad
 \bar\partial\ct_\Omega=\mathbf 1_\Omega.
\]
On $\C\setminus\overline\Omega$ the potential is harmonic and
\[
 \ct_\Omega=2\partial_z\LogPot_\Omega.
\]
Moreover, for $|z|$ sufficiently large,
\[
 \ct_\Omega(z)=\sum_{j=0}^{\infty}\frac{M_j(\Omega)}{z^{j+1}},
 \qquad
 M_j(\Omega)=\frac1\pi\int_\Omega \zeta^j\,dA(\zeta).
\]
Thus the exterior Cauchy transform is the generating function of the analytic moments of area measure.
\end{proposition}

\begin{proof}
The distributional identities follow from $\Delta\log|z|=2\pi\delta_0$ and $\bar\partial(1/(\pi z))=\delta_0$.  Differentiation under the integral gives the identity outside $\overline\Omega$, and the moment expansion follows by expanding $(z-\zeta)^{-1}$ as a geometric series.
\end{proof}

\begin{proposition}[Cauchy--Green]\label{prop:CG}
Orient $\bd\Omega$ positively. For $z\notin\overline\Omega$,
\[
 \ct_\Omega(z)=\frac1{2\pi i}\int_{\bd\Omega}
 \frac{\bar\zeta\,d\zeta}{z-\zeta}.
\]
More generally, if $\rho(z,\bar z)$ is polynomial and $\partial R_\rho/\partial w=\rho$, then
\[
 \frac1\pi\int_\Omega\frac{\rho(\zeta,\bar\zeta)dA(\zeta)}{z-\zeta}
 =\frac1{2\pi i}\int_{\bd\Omega}
 \frac{R_\rho(\zeta,\bar\zeta)d\zeta}{z-\zeta}.
\]
\end{proposition}

\begin{proof}
Apply Cauchy--Green to $\bar\zeta/(z-\zeta)$, respectively to $R_\rho(\zeta,\bar\zeta)/(z-\zeta)$.
\end{proof}

On the normalization $X$, the unweighted form is $f\,d\pi=\eta\,d\pi$; for the polynomial density it is $R_\rho(\pi,\eta)d\pi$. We now record the chain residue formula in the general form needed below.

\begin{proposition}[Chain residue formula]\label{prop:chain-res}
Let $f\in\C(X)$ and let $\gamma=\bd\Sigma$ for an integral two-chain on $X$, with $\gamma$ disjoint from the poles of $f\,d\pi$.  Let $n_\Sigma$ be the locally constant integer multiplicity of $\Sigma$ on $X\setminus|\gamma|$. For $z\notin\pi(\gamma)$,
\begin{equation}\label{eq:chain-res}
 \frac1{2\pi i}\int_\gamma\frac{f\,d\pi}{z-\pi}
 =\sum_{a\in X}n_\Sigma(a)\Res_a\frac{f\,d\pi}{z-\pi}.
\end{equation}
If $q\in\pi^{-1}(z)$ is a point at which $f$ is regular and $\pi$ has ramification index $e_q$, then the moving-pole residue equals $-e_qf(q)$.  In particular it is $-f(q)$ at a regular fibre point.  The sum of the residues at the fixed poles of $f\,d\pi$ is a rational function $R_\Sigma(z)$.
\end{proposition}

\begin{proof}
Triangulate $\Sigma$ so that no pole lies on an edge.  Applying the ordinary residue theorem to each two-simplex and summing cancels all interior edges, leaving the boundary $\gamma$; this gives \eqref{eq:chain-res}.  If $u$ is a local coordinate at $q$ with $\pi-z=u^{e_q}$, then
\[
 \frac{f\,d\pi}{z-\pi}=-e_q f(u)\frac{du}{u},
\]
which gives the stated moving-pole residue. At a fixed point $a$, expansion in a local coordinate shows that the residue is rational in the parameter $z$; only finitely many fixed points occur.
\end{proof}

\begin{remark}[Choice of the bounding chain]\label{rem:chain-choice}
Two integral two-chains with boundary $\gamma$ differ, up to a boundary, by an integral multiple of the fundamental class $[X]$.  Replacing $\Sigma$ by $\Sigma+N[X]$ adds $N(1,\ldots,1)$ to the multiplicity vector $m$ and adds $N\operatorname{Tr}_\pi(f)$ to $R_\Sigma$, where $\operatorname{Tr}_\pi(f)=\sum_{i=1}^d f(q_i)$ on a regular fibre; hence the difference in \eqref{eq:A-intro} and the orbit size $|G\cdot m|$ are unchanged.  Taking $\Sigma=X$ gives $R_X=\operatorname{Tr}_\pi(f)$, which explains the term ``incomplete trace.''
\end{remark}

\begin{remark}
The formula is a statement about the boundary-integral branch on each component of $\C\setminus\pi(\gamma)$. On the exterior component it equals the area Cauchy transform. On an interior component it is generally not the literal area transform.
\end{remark}

\section{The branch-sum cocycle}\label{sec:cocycle}

Let $B\subset\PP^1$ contain the branch values of $\pi$, the images of the poles of $f$, and infinity, and choose $b\in\C\setminus(B\cup\Gamma)$. Write
\[
 \pi^{-1}(b)=\{p_1,\ldots,p_d\},
 \qquad f_i=f(p_i)
\]
for the corresponding germs at $b$. Let
\[
 \rho:\pi_1(\PP^1\setminus B,b)\longrightarrow G\le S_d
\]
be the monodromy representation and let
\[
 M=\Z^d,\qquad M_0=\{(a_i)\in\Z^d:\sum_i a_i=0\}.
\]

For a loop $\ell$ transverse to $\Gamma$, the Plemelj crossing formula says that analytic continuation of
\[
 I_\gamma(z)=\frac1{2\pi i}\int_\gamma\frac{f\,d\pi}{z-\pi}
\]
changes it by a signed sum of branches of $f$. Recording the integer coefficient of each branch at $b$ gives a map
\[
 c_\gamma:\pi_1(\PP^1\setminus B,b)\longrightarrow M_0.
\]
The sum of the coordinates is zero because a closed loop has total signed intersection zero with the Jordan curve $\Gamma$.

\begin{lemma}[Intersection cocycle]\label{lem:cocycle}
Use the convention that $\ell_1\ell_2$ means that $\ell_2$ is traversed first and $\ell_1$ second, and let $\rho(\ell)$ transport endpoint sheet labels back to the base fibre.  With the left action
\[
 (ga)_j=a_{g^{-1}j},\qquad g\in G,
\]
the map $c_\gamma$ is a crossed homomorphism,
\begin{equation}\label{eq:cocycle}
 c_\gamma(\ell_1\ell_2)=c_\gamma(\ell_1)+\rho(\ell_1)c_\gamma(\ell_2).
\end{equation}
Moreover:
\begin{enumerate}[label=(\roman*)]
\item after a homotopy making $\ell$ transverse to $\Gamma$, let $\widetilde\ell_i^{\rm end}$ be the lift of $\ell$ ending at $p_i$; whenever this lift is closed,
\[
 c_{\gamma,i}(\ell)=\widetilde\ell_i^{\rm end}\cdot\gamma,
\]
up to the single global sign fixed by the Plemelj convention;
\item the combinatorial monodromy of $I_\gamma$ is
\[
 A_\gamma(\ell)=\ev_f(c_\gamma(\ell))
 =\sum_{i=1}^d c_{\gamma,i}(\ell)f_i.
\]
\end{enumerate}
\end{lemma}

\begin{proof}
Split a transverse loop $\ell$ at its crossings with $\Gamma$.  At a crossing point $a$, the branch carried by $\Gamma$ is the value of $f$ at the unique point of $\gamma$ over $a$, since $\pi|_\gamma$ is one-to-one. Continue this branch along the remaining part of $\ell$ to its endpoint at the base point.  The coefficient of $f_i$ is therefore the signed number of intersections of the lift ending at $p_i$ with $\gamma$.  This proves~(i), and applying $\ev_f$ gives exactly the sum of branches in~\cite[Equation~(4.1)]{PRY}, proving~(ii).

For $\ell_1\ell_2$, the crossings made during $\ell_1$, which is traversed last, contribute $c_\gamma(\ell_1)$.  The endpoint labels of the crossings made during $\ell_2$ are then transported through $\ell_1$, contributing $\rho(\ell_1)c_\gamma(\ell_2)$.  This proves \eqref{eq:cocycle}. Finally, $\sum_i c_{\gamma,i}(\ell)=0$, since it is the signed intersection number of the closed planar loop $\ell$ with the Jordan curve $\Gamma$.
\end{proof}

\begin{remark}[Relative rather than absolute monodromy]\label{rem:relative}
The cocycle $c_\gamma$ is a sheet-valued, relative crossing invariant.  The argument does not assert a Picard--Lefschetz increment in absolute homology: point-pushing of punctures may act trivially on $H_1(X,\Z)$ while the relative branch-sum cocycle is unbounded.  Lemma~\ref{lem:finite-null} instead recovers the absolute class $[\gamma]$ from the intersection numbers of closed lifts supplied by Lemma~\ref{lem:cocycle}(i).
\end{remark}

The integral above is a Cauchy-type integral of the algebraic branch of $f$ carried by the closed curve $\Gamma$, in the precise sense of~\cite{PRY}.  Since the branch is regular on $\Gamma$ and the curve has no endpoints or self-intersections, the algebraicity criterion of Pakovich--Roytvarf--Yomdin applies directly: $I_\gamma$ is algebraic if and only if the image of $A_\gamma$ is finite; see~\cite[Theorem~4.5]{PRY}. The next observation converts finiteness of the evaluated cocycle into topology.

\begin{lemma}[Finite cocycle implies null homology]\label{lem:finite-null}
Assume that $\ev_f:M_0\otimes\C\to L$ is injective. If the image of $A_\gamma$ is finite, then $[\gamma]=0$ in $H_1(X,\Z)$.
\end{lemma}

\begin{proof}
Because $\ev_f$ is injective, finiteness of the image of $A_\gamma=\ev_f\circ c_\gamma$ implies finiteness of the image of $c_\gamma$.  Put $K=\ker\rho$.  On $K$, equation~\eqref{eq:cocycle} reduces to
\[
 c_\gamma(k_1k_2)=c_\gamma(k_1)+c_\gamma(k_2),
\]
so $c_\gamma|_K$ is a homomorphism into the torsion-free abelian group $M_0$.  Its image is finite and therefore zero.  Hence $c_\gamma$ depends only on the monodromy permutation and descends to a crossed homomorphism
\[
 \bar c_\gamma:G\longrightarrow M_0.
\]

After tensoring with $\Q$, every crossed homomorphism of a finite group is a coboundary.  In the present convention this can be checked directly: set
\[
 S=\sum_{g\in G}\bar c_\gamma(g),\qquad m=\frac{S}{|G|}.
\]
Summing the cocycle identity $\bar c_\gamma(gh)=\bar c_\gamma(g)+g\bar c_\gamma(h)$ over $h\in G$ gives
\[
 \bar c_\gamma(g)=m-gm.
\]
If $g$ fixes the sheet $i$, then $(gm)_i=m_i$, and therefore
\[
 \bar c_{\gamma,i}(g)=0.                                      \tag{*}
\]

Now fix $p_i\in\pi^{-1}(b)$.  Let $\alpha$ be any loop on $X$, based at $p_i$, which avoids the finite set $\pi^{-1}(B)$.  Its projection $\ell=\pi\circ\alpha$ lies in $\PP^1\setminus B$, and the lift of $\ell$ from $p_i$ is precisely $\alpha$; in particular, $\rho(\ell)$ fixes $i$.  By~(*) and Lemma~\ref{lem:cocycle},
\[
 \alpha\cdot\gamma=0.
\]
Every class of $H_1(X,\Z)$ has such a representative: start with a cycle avoiding the finite set, join it to $p_i$, and conjugate it by the joining path.  Thus $[\gamma]$ has zero intersection with every class in $H_1(X,\Z)$.  Nondegeneracy of the intersection pairing on the compact oriented surface $X$ gives $[\gamma]=0$.
\end{proof}

\begin{lemma}[Irreducible augmentation gives injectivity]\label{lem:2trans}
If the augmentation representation $M_0\otimes\C$ of $G$ is irreducible and $f\notin\C(\pi)$, then
\[
 \ev_f:M_0\otimes\C\longrightarrow L
\]
is injective.
\end{lemma}

\begin{proof}
The kernel of $\ev_f$ is a $G$-submodule. Since $f\notin\C(\pi)$, not all conjugates $f_i$ are equal, so $\ev_f$ is nonzero on the augmentation module. Its kernel is therefore zero.
\end{proof}

\section{Proofs of the main criteria}\label{sec:proofs}

\begin{proof}[Proof of Theorem~\ref{thm:A}]
Proposition~\ref{prop:chain-res} gives
\[
 I_\gamma(z)=R_\Sigma(z)-\sum_i m_i f_i(z).
\]
Let $L$ be the Galois closure and $G=\operatorname{Gal}(L/\C(z))$, identified with the geometric monodromy group acting on the sheets. All conjugates of $I_\gamma$ are among
\[
 R_\Sigma(z)-\sum_i (gm)_i f_i(z),\qquad g\in G,
\]
so the degree is at most $|G\cdot m|$. Equivalently, an explicit annihilating polynomial is
\begin{equation}\label{eq:norm-poly}
 Q(z,Y)=\prod_{\mu\in G\cdot m}
 \left(Y-R_\Sigma(z)+\sum_i\mu_i f_i(z)\right)\in\C(z)[Y],
\end{equation}
where the product is taken over the orbit vectors; the resulting product is $G$-invariant.

If $gm\ne hm$, then $gm-hm\in M_0$. Injectivity of $\ev_f$ implies
\[
 \sum_i(gm-hm)_if_i\ne0,
\]
so all orbit expressions are distinct and the degree is $|G\cdot m|$. If $G=S_d$ and $f\notin\C(\pi)$, then the augmentation representation is irreducible, so Lemma~\ref{lem:2trans} supplies the required injectivity. The orbit size of a vector whose equal entries occur with multiplicities $\nu_1,\ldots,\nu_s$ is the stated multinomial coefficient. The $0$--$1$ case gives $\binom dk$.
\end{proof}

\begin{corollary}[The rational case inside the orbit formula]\label{cor:rational-orbit}
Under the hypotheses of Theorem~\ref{thm:A}, assume that $\ev_f$ is injective on the augmentation module.  Then $I_{\gamma,f}$ is rational if and only if the multiplicity vector $m$ is constant.  For a bordered side, this is equivalent to $k\in\{0,d\}$.  In the exterior Cauchy-transform setting, under the standard boundary regularity assumptions, this is precisely the quadrature-domain case.
\end{corollary}

\begin{proof}
By Theorem~\ref{thm:A}, the algebraic degree is $|G\cdot m|$.  It equals one exactly when $m$ is $G$-invariant.  Since $G$ is transitive, its invariant vectors in $\C^d$ are the constant vectors.  For a bordered side, $m$ is a $0$--$1$ vector, so it is constant exactly when $k=0$ or $k=d$.
\end{proof}

\begin{proof}[Proof of Theorem~\ref{thm:B}]
If $[\gamma]=0$, then the simple closed curve $\gamma$ separates $X$ and bounds an integral two-chain. Proposition~\ref{prop:chain-res} shows that $I_\gamma$ is algebraic.

Conversely, suppose that $I_\gamma$ is algebraic. By~\cite[Theorem~4.5]{PRY}, its combinatorial monodromy $A_\gamma$ has finite image. The augmentation representation is irreducible and $f\notin\C(\pi)$, so Lemma~\ref{lem:2trans} applies. Lemma~\ref{lem:finite-null} then gives $[\gamma]=0$.
\end{proof}

\begin{remark}[Sharper evaluation hypotheses]\label{rem:sharp-evaluation}
Two-transitivity is a convenient verifiable sufficient condition.  Global injectivity of $\ev_f$ on the augmentation module is equivalent to linear independence in $L$ of
\[
 f_1-f_d,\ldots,f_{d-1}-f_d.
\]
The proof of Theorem~\ref{thm:B} needs still less: it suffices that $\ev_f$ be injective on the $G$-submodule of $M_0\otimes\C$ generated by the image of $c_\gamma$.  Thus the criterion can apply even when the full augmentation representation is reducible.
\end{remark}

\section{Consequences for real algebraic ovals}\label{sec:real}

\subsection{Generic projections}

Let the relevant projective plane curve be nonsingular. For a generic affine linear function, all finite critical points are simple and their critical values are distinct. The corresponding local monodromies are transpositions. Since the curve with the poles of $\pi$ and the ramification points removed is connected, these finite branch transpositions generate a transitive subgroup of $S_d$; their transposition graph is therefore connected, and the generated group is the full symmetric group $S_d$. This argument does not require the ramification over $\infty$ to be simple. Real affine coordinates form a real Zariski-dense family of complex affine linear functions, so the same conclusion holds after a generic real affine change of coordinates. For singular plane models we make no genericity assertion here; Theorem~\ref{thm:B} still applies whenever the actual projection monodromy has irreducible augmentation representation.

\begin{proof}[Proof of Corollary~\ref{cor:generic}]
For a generic real affine coordinate, the projection $\pi$ has monodromy $S_d$. Theorem~\ref{thm:B} gives the equivalence between algebraicity and separation.

A nonsingular real cubic with a compact affine oval has two real components on its elliptic normalization: the compact oval and the projective pseudoline; see, for example,~\cite{GrossHarris}. Each component is nonseparating, so the compact oval has generically nonalgebraic transform.

More generally, if $X(\R)$ has at least two components, no individual component separates $X$. Indeed, if one component separated, conjugation would interchange the two sides and no further fixed point could exist.

Finally, let a nonsingular real plane curve of degree $2k$ have exactly one oval. If it were dividing, Rokhlin's complex-orientation formula~\cite{Rokhlin}
\[
 2(\Pi_+-\Pi_-)=l-k^2
\]
would apply, where $l$ is the number of ovals and $\Pi_+$ and $\Pi_-$ are the numbers of positively and negatively nested pairs.  Here it gives $0=1-k^2$, because $l=1$ and there are no nested pairs. Hence $k=1$. For degree at least four the unique oval is nonseparating, and the generic nonalgebraicity conclusion follows.
\end{proof}

\subsection{Positive results}

\begin{corollary}[Rational normalization]\label{cor:rational}
Let $\Gamma$ be a smooth compact oval contained in the nonsingular locus of a real plane curve, and assume that the relevant irreducible component has rational normalization. Then its exterior Cauchy transform, and every polynomially weighted exterior transform, is algebraic.
\end{corollary}

\begin{proof}
A closed curve on $\PP^1$ bounds an integral two-chain. Apply Proposition~\ref{prop:chain-res} with $f=\eta$, or with $f=R_\rho(\pi,\eta)$.
\end{proof}

\begin{corollary}[Complete real locus of a dividing curve]\label{cor:dividing}
Suppose $X$ is dividing and the lifted boundary of a bounded plane domain is the complete real locus $X(\R)$ with the complex orientation of one half of $X\setminus X(\R)$. Then the exterior Cauchy transform and every polynomially weighted exterior transform are algebraic.
\end{corollary}

\begin{proof}
Choose the half of $X\setminus X(\R)$ whose induced boundary orientation agrees with the positive orientation of the plane boundary. It is an integral two-chain with boundary $X(\R)$. Choosing the other half changes the sign of the boundary integral and does not affect algebraicity.
\end{proof}

\begin{remark}
For an affine complete real locus to be a compact planar boundary, the projective degree must be even and the affine chart must avoid real points at infinity on the chosen boundary. The statement concerns the complete oriented real locus; an individual component of a multi-component dividing curve is nonseparating.
\end{remark}

\subsection{The ellipse as a sharp degree example}

For the ellipse with semiaxes $a>b>0$ and $c^2=a^2-b^2$,
\[
 \ct_\Omega(z)=\frac{2ab}{c^2}\left(z-\sqrt{z^2-c^2}\right),
 \qquad \sqrt{z^2-c^2}\sim z\quad(z\to\infty).
\]
This is also a direct check of the incomplete-trace formula.  Put
\[
 A=\frac{a+b}{2},\qquad B=\frac{a-b}{2},
\]
and parametrize the Schwarz conic by
\[
 \pi(\lambda)=A\lambda+B\lambda^{-1},
 \qquad
 \eta(\lambda)=B\lambda+A\lambda^{-1}.
\]
For exterior $z$, the two roots of $\pi(\lambda)=z$ satisfy $|\lambda_1\lambda_2|=B/A<1$, and exactly one, say $\lambda_2$, lies in $\Sigma=\{|\lambda|<1\}$.  The fixed pole at $\lambda=0$ contributes
\[
 R_\Sigma(z)=\frac ABz=\frac{a+b}{a-b}z,
\]
and a direct simplification gives
\[
 R_\Sigma(z)-\eta(\lambda_2)
 =\frac{2ab}{c^2}\left(z-\sqrt{z^2-c^2}\right).
\]
Here $d=2$, the monodromy is $S_2$, and $k=1$, so Theorem~\ref{thm:A} gives the exact degree $\binom21=2$; Corollary~\ref{cor:rational-orbit} simultaneously shows why the transform is not rational.

\section{Explicit positive-genus domains beyond quadrature domains}\label{sec:examples}

Fix $h\ge1$ and put
\[
 n=2h+1.
\]
Let $X_h$ be the smooth projective normalization of
\begin{equation}\label{eq:hyp}
 y^2=x^{2n}+1.
\end{equation}
It is hyperelliptic of genus $n-1=2h$. Since $x^{2n}+1>0$ on $\R$, the real locus has no finite branch point. There are two real points over infinity.

Define
\begin{equation}\label{eq:tr}
 t=y+x^n,\qquad t^{-1}=y-x^n,
 \qquad r=\frac{x}{1+x^2}.
\end{equation}
The identity for $t^{-1}$ follows from $(y+x^n)(y-x^n)=1$.  Globally on $X_h$,
\[
 x^n=\frac{t-t^{-1}}2,
 \qquad y=\frac{t+t^{-1}}2.
\]
Since $n$ is odd, $t$ identifies $X_h(\R)$ bijectively with $\R\PP^1$. Thus the real locus has one component. It is dividing: the preimages of the upper and lower half-planes under $x:X_h\to\PP^1$ are connected and form the two components of $X_h\setminus X_h(\R)$. Connectivity follows because each half-plane contains at least one branch point of the hyperelliptic double cover.

Put
\[
 U(t)=\frac{t^2-1}{t^2+1},
 \qquad V(t)=\frac{2t}{t^2+1},
 \qquad s=1+\varepsilon r,
\]
and define real meromorphic functions
\begin{equation}\label{eq:uv}
 u=sU(t),\qquad v=sV(t),
 \qquad z=u+iv,\qquad w=u-iv.
\end{equation}
The rational identity
\[
 (U+iV)(U-iV)=U^2+V^2=1
\]
holds on all of $X_h$.  Consequently
\begin{equation}\label{eq:zw-squared}
 zw=s^2,\qquad w=\frac{s^2}{z}.
\end{equation}
On $X_h(\R)$ one has $|r|\le1/2$, so $s>0$ when $|\varepsilon|<2$.

\begin{lemma}[The boundary is a smooth star-shaped Jordan curve]\label{lem:smooth}
Let $0<|\varepsilon|<2$. Then $z|_{X_h(\R)}$ is an injective immersion. Hence
$\Gamma_{h,\varepsilon}:=z(X_h(\R))$ is a real-analytic Jordan curve, star-shaped
with respect to the origin, and $z:X_h(\R)\to\Gamma_{h,\varepsilon}$ is a
diffeomorphism.
\end{lemma}

\begin{proof}
Put
\[
 \phi=\frac{t+i}{t-i},
\]
so that $U+iV=\phi$ and $z=s\phi$. On $X_h(\R)$ the function $t$ is real, hence $|\phi|=1$; write $\phi=e^{i\theta}$. Since $t\mapsto\phi$ is a M\"obius transformation carrying $\R\PP^1$ diffeomorphically onto the unit circle, and $t$ identifies $X_h(\R)$ with $\R\PP^1$ bijectively, the map $z=se^{i\theta}$ with $s>0$ is injective, and its image meets each ray from the origin in exactly one point. Thus the image is star-shaped.

For immersivity, $s$ and $\theta$ are real valued on $X_h(\R)$, and
\[
 dz=\phi\bigl(ds+is\,d\theta\bigr).
\]
Comparing real and imaginary parts, $dz$ can vanish only if $ds=d\theta=0$. Now $d\theta=0$ precisely at the four ramification points of $t$: the two points with $x=0$, where $t=\pm1$, and the two points over $x=\infty$, where $t=0,\infty$. At each of them $t$ is totally ramified of index $n$. At a point with $x=0$, the function $x$ is a local coordinate and
\[
 ds=\varepsilon\,dr
 =\varepsilon\frac{1-x^2}{(1+x^2)^2}\,dx
 =\varepsilon\,dx\ne0.
\]
At a point over $x=\infty$, the function $\xi=1/x$ is a local coordinate, $r=\xi/(1+\xi^2)$, and $ds=\varepsilon\,d\xi\ne0$. Hence $dz\ne0$ everywhere on $X_h(\R)$.
\end{proof}

\begin{remark}\label{rem:polar-regularity}
The condition $\varepsilon\ne0$ enters both Lemma~\ref{lem:birational} and Lemma~\ref{lem:smooth}. In the latter it is exactly what makes $z$ an immersion at the four real ramification points of $t$; for $\varepsilon=0$, $dz$ vanishes there to order $n-1$.  Because $t$ ramifies at these points, $s$ is continuous but not $C^1$ as a function of the polar angle: locally $s-1\sim c(\theta-\theta_0)^{1/n}$. Thus $\Gamma_{h,\varepsilon}$ is a smooth curve tangent to the radial direction at four points, not the graph of a smooth function of $\theta$.
\end{remark}

\begin{lemma}[Birationality of the explicit projection]\label{lem:birational}
For every $\varepsilon\ne0$, the pair $(u,v)$ generates $\C(X_h)$. Consequently the irreducible plane curve traced by $(z,w)$ has normalization $X_h$.
\end{lemma}

\begin{proof}
From \eqref{eq:tr},
\[
 \C(X_h)=\C(t,x),\qquad x^n=\frac{t-t^{-1}}2.
\]
The function $(t-t^{-1})/2$ has simple zeros and poles, hence is not an $n$th power in $\C(t)$. Thus the Kummer extension over $\C(t)$ has degree $n$, with conjugates $x\mapsto\zeta x$, $\zeta^n=1$. The conjugates
\[
 \frac{\zeta x}{1+\zeta^2x^2}
\]
of $r$ are pairwise distinct. Indeed, equality for $\zeta_1$ and $\zeta_2$ gives
\[
 (\zeta_1-\zeta_2)(1-\zeta_1\zeta_2x^2)=0,
\]
and $x^2$ is nonconstant in $\C(X_h)$; hence $\zeta_1=\zeta_2$. Therefore $\C(t,r)=\C(t,x)=\C(X_h)$.

The quotient
\[
 \frac vu=\frac{V(t)}{U(t)}=\frac{2t}{t^2-1}
\]
shows that $t$ has degree at most two over $\C(u,v)$. Once $t$ is adjoined, $s=u/U(t)$ and hence $r=(s-1)/\varepsilon$ are recovered. Therefore
\[
 [\C(X_h):\C(u,v)]\le2.
\]
If this degree were two, its nontrivial automorphism would send $t$ to $-t^{-1}$; every lift to $X_h$ has the form
\[
 t\mapsto-t^{-1},\qquad x\mapsto\zeta x,
 \qquad \zeta^n=1.
\]
Under this map both $U$ and $V$ change sign. Fixing $u$ and $v$ would force
\[
 1+\varepsilon\frac{\zeta x}{1+\zeta^2x^2}
 =-1-\varepsilon\frac{x}{1+x^2}
\]
as an identity on $X_h$. At $x=0$ the two sides are $1$ and $-1$, a contradiction. Hence the degree is one.
\end{proof}

\begin{proof}[Proof of Theorem~\ref{thm:C}]
Let $\Gamma_{h,\varepsilon}$ be the real-analytic Jordan curve supplied by Lemma~\ref{lem:smooth}, and let $\Omega_{h,\varepsilon}$ be its bounded interior. Elimination of $x,y$ gives an irreducible algebraic relation $P(z,w)=0$, and on the real locus $w=\bar z$. Lemma~\ref{lem:birational} shows that the normalization of the relevant Schwarz component is $X_h$, of genus $2h$.

The real locus is the boundary of either half of the dividing curve $X_h$. Choose the half whose induced boundary orientation maps to the positive orientation of $\Gamma_{h,\varepsilon}$. Proposition~\ref{prop:chain-res} then proves that $\ct_{\Omega_{h,\varepsilon}}$ is algebraic; choosing the other half would merely change the sign of the boundary integral.

Suppose it were rational. By the standard characterization of quadrature domains~\cite{AharonovShapiro,Sakai,GustafssonShapiro}, the simply connected domain $\Omega_{h,\varepsilon}$ would be a quadrature domain. Its Schwarz correspondence along the analytic boundary would then have rational normalization. That rational component and the irreducible component parametrized by $X_h$ contain the same real-analytic boundary arc, so they coincide. This contradicts the positive genus of $X_h$. Thus the transform is nonrational.
\end{proof}

\begin{remark}[Optimal parity]
For a dividing real curve of genus $g$ with $r$ real components, Klein's congruence (see~\cite{GrossHarris}) gives $r\equiv g+1\pmod2$. Hence a dividing curve with one real component has even genus. The family above realizes every positive genus allowed by this constraint.
\end{remark}

\section{Concluding remarks}

The general theorem of Pakovich--Roytvarf--Yomdin characterizes algebraicity of a Cauchy-type integral through finite combinatorial monodromy \cite{PRY}.  Theorem~\ref{thm:B} evaluates that criterion geometrically in the oval setting.  Under irreducible augmentation monodromy, the evaluated branch-sum cocycle is finite precisely when the lifted boundary has zero homology class.  Thus the obstruction to algebraicity is global: it is the intersection cocycle of the lifted oval with closed curves on the normalization.

From the potential-theoretic viewpoint, classical quadrature domains are exactly the domains for which the exterior Cauchy transform is rational, under the usual regularity hypotheses.  Theorem~\ref{thm:C} shows that replacing rationality by algebraicity produces a much larger class: there are simply connected real-algebraic domains whose exterior Cauchy transforms have finite nontrivial branching and whose Schwarz normalizations have arbitrary positive even genus.  Theorem~\ref{thm:A} makes this finite branching effective by identifying the transform with an incomplete trace and computing its degree from the monodromy orbit of the side multiplicity vector.

The principal contribution of the paper is therefore the equivalence
\[
 \begin{gathered}
 \text{algebraicity of the exterior Cauchy transform}\\
 \Longleftrightarrow\\[-2pt]
 \text{separation by the lifted oval}.
 \end{gathered}
\]
valid for the broad, and generically occurring, class of projections with two-transitive monodromy.  It converts a finiteness property of an exterior logarithmic potential into a topological property of the compact Schwarz correspondence.

\section*{Data availability}
No datasets were generated or analysed during the current study.

\section*{Statements and declarations}
\noindent\textbf{Funding.} The authors declare that no funds, grants, or other support were received during the preparation of this manuscript.\par
\noindent\textbf{Competing interests.} The authors have no relevant financial or non-financial interests to disclose.\par
\noindent\textbf{Author contributions.} Both authors contributed to the conception of the study, the mathematical development, and the writing of the manuscript. Both authors read and approved the final manuscript.\par

\section*{Acknowledgements} 
 The second author is sincerely grateful to Victor Vassiliev for his interest and discussions of this topic.  He also thanks Eugene Shustin for explanations of the Rokhlin--Mishachev complex-orientation formula. GPT 5.6 Sol has been used to polish the original manuscript and strengthen several formulations. All AI-generated proofs have been manually checked.

\end{document}